\documentclass{amsart}

\usepackage{amssymb}
\usepackage{color}

\newcommand{\mfT}{{\mathfrak T}}
\newcommand{\R}{\mathbb R}

\newcommand{\dsp}{\displaystyle}
\newtheorem{proposition}{Proposition}
\newtheorem{theorem}{Theorem}
\newtheorem{definition}{Definition}
\newtheorem{remark}{Remark}

\newtheorem{lemma}{Lemma}

\title[Derivation and analysis of A NEW $2D$ Green-Naghdi system]{Derivation and analysis of A NEW $2D$ Green-Naghdi system}
\author{Samer Israwi}
\address{Universit\'e Bordeaux I; IMB, et CNRS, UMR 5251, 351 Cours de la Lib\'eration, 33405 Talence Cedex, France}
\thanks{Ce travail a b\'en\'efici\'e d'une aide de l'Agence Nationale de la Recherche portant  la r\'ef\'erence ANR-08-BLAN-0301-01}
\email{Samer.Israwi@math.u-bordeaux1.fr}

\begin{document}

\maketitle
\begin{abstract}
We derive here a variant of the $2D$ Green-Naghdi  equations that model the propagation of   two-directional, nonlinear dispersive waves
in  shallow water. This new model has the same accuracy  as the standard $2D $ Green-Naghdi equations. 
Its mathematical interest is that it allows a control of the rotational part of the  (vertically averaged) horizontal velocity, which is not the case for  the usual Green-Naghdi equations. Using
this property, we
show that  the solution of these new  equations  can be constructed by a 
standard Picard iterative scheme so that there is no loss of regularity of the solution
with respect to the initial condition. Finally, we prove that the new Green-Naghdi  equations
 conserve the almost irrotationality of the vertically averaged horizontal component of the velocity.
\end{abstract}
\section{Introduction}
\subsection{General setting}
 The water-waves problem,  consists in studying the motion of the 
free surface and the evolution of the velocity field of a layer of fluid under the following assumptions:
the fluid is ideal, incompressible, irrotationnal, and under the only influence of gravity.
 Many works have set a good theoretical background for this problem. 
 Its local well-posedness has been discussed among others by Nalimov \cite{nalimov}, Yosihara \cite{yoshihara}, Craig \cite{craig}, Wu \cite{wu1,wu2} and Lannes \cite{lannes}. 
  Since we are interested here in the asymptotic behavior of the solutions, it is convenient to work with a non-dimensionalized version of the equations. 
 In this framework, (see for instance \cite{AL} for details), the free surface is 
parametrized by   $z=\varepsilon\zeta(t,X)$  (with $X=(x,y)\in \R^2$)
and the bottom by $z=-1+\beta b(X)$. 
Here, $\varepsilon$ and $\beta$ are dimensionless parameters  defined as 

\begin{equation*}
\varepsilon=\frac{a_{surf}}{h_0},\quad
\beta=\frac{b_{bott}}{h_0};
\end{equation*}
where  $a_{surf}$ is the typical  amplitude of the waves
and $b_{bott}$
is the typical amplitude of the bottom deformations,
while $h_0$ is the  depth. 
One can use the incompressibility and
irrotationality  conditions to write  the non-dimensionalized water-waves  equations under
Bernoulli's formulation,  in terms of a  velocity potential $\varphi$
associated to the flow, where $\varphi(t, .)$ is defined on $\Omega_t
= \{(X,z), -1  +\beta b(x) < z < \varepsilon\zeta(t,X)\}$ (i.e. the velocity
field is given by $U=\nabla_{X,z}\varphi$) :
\smallskip
\begin{equation}
\left\{
\begin{array}{lcl}
\dsp\mu\partial_x^2 \varphi+\mu\partial_y^2 \varphi + \partial_z^2 \varphi = 0, & \hbox{at} & -1+\beta b< z<\varepsilon \zeta ,\vspace{1mm}\\
\dsp\partial_z\varphi-\mu\beta\nabla b\cdot\nabla\varphi=0& \hbox{at} & z = -1 + \beta b,\vspace{1mm}\\
\dsp\partial_t \zeta -\frac{1}{\mu}(\mu\varepsilon\nabla \zeta\cdot\nabla \varphi + \partial_z \varphi)=0, & \hbox{at} & z=\varepsilon\zeta,\vspace{1mm}\\
\dsp\partial_t \varphi + \frac{1}{2}(\varepsilon\vert\nabla \varphi\vert^2 + \frac{\varepsilon}{\mu}(\partial_z \varphi)^2) + \zeta= 0 & \hbox{at} & z=\varepsilon \zeta,
\end{array}
\right.
\label{wwadim1}
\end{equation}
where $\nabla=(\partial_x,\partial_y)^T=(\partial_1,\partial_2)^T$. The \emph{shallowness}  parameter  $\mu$ appearing in this set of equations is defined as 
\begin{equation*}
 \mu=\frac{h_0^2}{\lambda^2}, 
 \end{equation*}
where,  $\lambda$ is the typical wave-length of the waves.
Making assumptions on the   size of $\varepsilon$, $\beta$, and
$\mu$ one is led to derive (simpler) asymptotic models from (\ref{wwadim1}). In
the  shallow-water  scaling  $(\mu\ll1)$,  one  can  derive  (when no smallness
assumption is made on  $\varepsilon$ and $\beta$ ) the standard Green-Naghdi equations (see \S \ref{pres1} below for a derivation and \cite{AL} for a
rigorous justification). For two-dimensional surfaces and over uneven
bottoms these  equations couple the free surface  elevation $\zeta$ to
the vertically averaged horizontal component of the velocity,
\begin{equation}
v(t,X)=\frac{1}{1+\varepsilon\zeta-\beta b}\int_{-1+\beta b}^{\varepsilon\zeta}\nabla\varphi(t,X,z)dz,
\label{averaged}
\end{equation}
and can be written as:
 \begin{equation}
\left\{
\begin{array}{lcl}
\dsp\partial_t\zeta+\nabla\cdot (hv)=0,\vspace{1mm}\\
\dsp\big(h+\mu\mathcal{T}[h,\beta b]\big)\partial_t v+h\nabla\zeta+\varepsilon \dsp\big(h+\mu\mathcal{T}[h,\beta b]\big)
( v\cdot \nabla )v\vspace{1mm}\\
\dsp\indent+\mu\varepsilon\big\lbrace\frac{2}{3}\nabla[(h^3(\partial_1v\cdot\partial_2v^\perp+(\nabla\cdot v)^2)]+ \Re[h,\beta b](v)\big\rbrace=0,
\end{array}
\right.
\label{GN}
\end{equation}
where $h=1+\varepsilon\zeta- \beta b$, $v=(V_1,V_2)^T$, $v^\perp=(-V_2,V_1)^\perp$, and 
\begin{eqnarray}\label{defoper}
\mathcal{T}[h,\varepsilon b]W
&=&-\frac{1}{3}\nabla(h^3\nabla\cdot W)+\frac{\beta}{2}[\nabla(h^2\nabla b\cdot W)-h^2 \nabla b\nabla\cdot W]\\&&+\beta^2 h\nabla b \nabla b\cdot W\nonumber,
\end{eqnarray}
while the purely topographic term $\Re[h,\beta b](v)$ is defined as:
\begin{eqnarray}\label{defoper1}
\Re[h,\beta b](v)&=&\frac{\beta}{2}\nabla\big(h^2(V_1^2\partial_1^2b+2V_1V_2\partial_1\partial_2 b+V_2^2\partial_2^2b)\big)\\
& &\nonumber+\beta h^2\big(\partial_1v\cdot\partial_2v^\perp+(\nabla\cdot v)^2\big)\nabla b\\
& &+\beta^2 h\big(V_1^2\partial_1^2b+2V_1V_2\partial_1\partial_2 b+V_2^2\partial_2^2b\big)\nabla b\nonumber.
\end{eqnarray}

A rigorous justification of the standard GN model was given by Li \cite{li} in $1D$ and for flat bottoms,  and by B. Alvarez-Samaniego and D. Lannes  \cite{AL}
 in the general case. This latter reference relies on  well-posedness results for these equations given in \cite{AL2} and based on general well-posedness results 
for evolution equations using a Nash-Moser scheme. The result of \cite{AL2} covers both the case of $1D$ and $2D$
surfaces, and allows for non flat bottoms. The reason why  a Nash-Moser scheme is used there is because
the estimates on the linearized equations exhibit losses of derivatives. However, in the $1D$ case, 
such losses do not occur and it is possible to construct a solution with a standard Picard iterative scheme
as in \cite{li,SI1} with flat and non flat bottoms respectively. But, this is not the case in $2D$, since  for instance, the term $\partial_1v\cdot\partial_2v^\perp$
is not controled by the energy norm $\vert \cdot\vert_{Y^s}$ naturally associated to (\ref{GN}) (see \cite{AL2}),
$$
\vert (\zeta,v)\vert_{Y^s}^2=\vert\zeta\vert_{H^s}^2+\vert v\vert_{H^s}^2
+\mu \vert \nabla\cdot v\vert_{H^s}^2.
$$ 

This is the motivation for the present derivation of a new variant  of the $2D $ Green-Naghdi  equations (\ref{GN}). This variant has  the same accuracy as the standard $2D $ Green-Naghdi equations (\ref{GN})
 (see \S \ref{pres2} below for a derivation) and can be written under the form:
 \begin{equation}
\left\{
\begin{array}{lcl}
\dsp\partial_t\zeta+\nabla\cdot (hv)=0,\vspace{1mm}\\
\dsp\Big(h+\mu\big(\mathcal{T}[h,\beta b]-\nabla^\perp\hbox{curl}\;\big)\Big)\partial_t v+h\nabla\zeta\vspace{1mm}\\\dsp\indent+\varepsilon \dsp\Big(h+\mu\big(\mathcal{T}[h,\beta b]-\nabla^\perp\hbox{curl}\,\;\big)\Big)
( v\cdot \nabla )v\vspace{1mm}\\
\dsp\indent+\mu\varepsilon\Big\lbrace\frac{2}{3}\nabla[h^3(\partial_1v\cdot\partial_2v^\perp+(\nabla\cdot v)^2)]+ \Re[h,\beta b](v)\Big\rbrace=0,
\end{array}
\right.
\label{GNnew}
\end{equation}
where $\nabla^\perp=(-\partial_y,\partial_x)^T$, $\hbox{curl}\,v=\partial_1 V_2-\partial_2 V_1$, 
while the linear operators  $\mathcal{T}[h,\varepsilon b]$ and $\Re[h,\beta b]$ are defined in (\ref{defoper}) and (\ref{defoper1}).\\
The reason why the new terms (involving $\nabla^\perp\hbox{curl}$ ) do not affect the precision of the model is because the solutions to (\ref{GN}) are nearly irrotational (in the sense that $\hbox{curl}\, v$ is small). This property is of course 
satisfied also by our new model (\ref{GNnew}). The presence of these new terms allows the definition of a new energy norm that controls  also the rotational part of $v$. Consequently, 
we show that  it is possible to use a standard   Picard iterative scheme to prove the well-posedness of (\ref{GNnew}),
so that there is no loss of regularity of the solution with respect to the initial condition.  
\subsection{Organization of the paper}
 We first recall the derivation of the standard $2D$ Green-Naghdi equations in Section \ref{pres1} while in Section  \ref{pres2} we derive the  new model (\ref{GNnew}).
  We give some preliminary results  in Section \ref{prnew};  the main theorem 
which proves the well-posedness of this new Green-Naghdi system is
then stated in Section \ref{lanew} and proved in Section \ref{mrnew}. Finally, in Section \ref{conservirrot}  we prove that  (\ref{GNnew})
 conserves the almost irrotationality of $v$.
\subsection{Notation}
We denote by $C(\lambda_1, \lambda_2,...)$ a constant depending on the parameters 
$\lambda_1$, $\lambda_2$, ... and \emph{whose dependence on the $\lambda_j$ is always assumed to be nondecreasing}.\\
The notation $a\lesssim b$ means that $a\leq Cb$, for some nonegative constant $C$ whose exact expression is of no importance (\emph{in particular, it is independent of the small parameters involved}).\\
Let $p$ be any constant
with $1\leq p< \infty$ and denote $L^p=L^p(\R^2)$ the space of all Lebesgue-measurable functions 
$f$ with the standard norm $$\vert f \vert_{L^p}=\big(\int_{\R^2}\vert f(X)\vert^p dX\big)^{1/p}<\infty.$$ When $p=2$,
we denote the norm $\vert\cdot\vert_{L^2}$ simply by $\vert\cdot\vert_2$. The inner product of any functions $f_1$
and $f_2$ in the Hilbert space $L^2(\R^2)$ is denoted by
$$
(f_1,f_2)=\int_{\R^2}f_1(X)f_2(X) dX.
$$  
The space $L^\infty=L^\infty(\R^2)$ consists of all essentially bounded, Lebesgue-measurable functions
$f$ with the norm
$$
\vert f\vert_{L^\infty}= \hbox{ess}\sup \vert f(X)\vert<\infty.
$$
We denote by $W^{1,\infty}=W^{1,\infty}(\R^2)=\{f\in L^\infty, \nabla f\in (L^{\infty})^2\}$ endowed with its canonical norm.\\
For any real constant $s$, $H^s=H^s(\R^2)$ denotes the Sobolev space of all tempered
distributions $f$ with the norm $\vert f\vert_{H^s}=\vert \Lambda^s f\vert_2 < \infty$, where $\Lambda$ 
is the pseudo-differential operator $\Lambda=(1-\Delta)^{1/2}$.\\
For any functions $u=u(x,t)$ and $v(x,t)$
defined on $\R^2\times [0,T)$ with $T>0$, we denote the inner product, the $L^p$-norm and especially
the $L^2$-norm, as well as the Sobolev norm,
with respect to the spatial variable $X$, by $(u,v)=(u(\cdot,t),v(\cdot,t))$, $\vert u \vert_{L^p}=\vert u(\cdot,t)\vert_{L^p}$, 
$\vert u \vert_{L^2}=\vert u(\cdot,t)\vert_{L^2}$ , and $ \vert u \vert_{H^s}=\vert u(\cdot,t)\vert_{H^s}$, respectively.\\
Let $C^k(\R^2)$ denote the space of  
$k$-times continuously differentiable functions and $C^{\infty}_0(\R^2)$ denote the space of infinitely differentiable
functions, with compact support in $\R^2$; we also denote by $C^\infty_b(\R^2)$ the space of infinitely differentiable functions that are bounded together with all their derivatives.\\
For any closed operator $T$ defined on a Banach space $X$ of functions, the commutator $[T,f]$ is defined
 by $[T,f]g=T(fg)-fT(g)$ with $f$, $g$ and $fg$ belonging to the domain of $T$.\\
 We denote  $v^\perp=(-V_2,V_1)^T$ and $\hbox{curl}\,v=\partial_1 V_2-\partial_2 V_1$ where $v=(V_1,V_2)^T$.
 \section{Derivation of the new Green-Naghdi model}\label{pres}

 This section is devoted to the derivation of a new Green-Naghdi  asymptotic model for the water-waves equations  in the shallow water ($\mu\ll 1$)
 of the same accuracy as the standard $2D $ Green-Naghdi equations (\ref{GN}).
 \subsection{Derivation of the standard Green-Naghdi equations (\ref{GN})}\label{pres1}
 We recall here the main steps of \cite{LB} for the derivation of  the standard $2D $ Green-Naghdi equations (\ref{GN}). 
 In order to reduce the model (\ref{wwadim1}) into a model of two equations we introduce the trace of the velocity potential at the free surface, defined as
$$
\psi=\varphi_{\mid_{z=\varepsilon\zeta}},
$$
and the Dirichlet-Neumann operator  $\mathcal{G}_{\mu}[\varepsilon\zeta,\beta b]\cdot$  as
$$
\mathcal{G}_{\mu}[\varepsilon\zeta,\beta b]\psi=-\mu\varepsilon\nabla\zeta\cdot\nabla\varphi_{\mid_{z=\varepsilon\zeta}} + \partial_z \varphi_{\mid_{z=\varepsilon\zeta}},
$$
with $\varphi$ solving the boundary value problem
\begin{equation}\label{bvp}
\left\{
\begin{array}{lcl}
\dsp\mu\partial_x^2 \varphi+\mu\partial_y^2 \varphi + \partial_z^2 \varphi = 0,\\
\dsp\partial_n\varphi_{\mid_{z=-1+\beta b}}=0,\\
\dsp\varphi_{\mid_{z=\varepsilon\zeta}}=\psi.
\end{array}
\right.
\end{equation}
 As remarked in \cite{zakarov,craigsulem}, the equations (\ref{wwadim1}) are equivalent to a set of two equations on the free surface parametrization $\zeta$ and the
trace of the velocity potential at the surface $\psi=\varphi_{\mid_{z=\varepsilon\zeta}}$ involving the Dirichlet-Neumann operator. Namely
\begin{equation}
\left\{
\begin{array}{lcl}
\dsp\partial_t \zeta+\frac{1}{\mu}\mathcal{G}_{\mu}[\varepsilon\zeta,\beta b]\psi= 0,\\
\dsp\partial_t\psi+\zeta+\frac{\varepsilon}{2}\vert\nabla\psi\vert^2
-\varepsilon\mu\dsp\frac{(\frac{1}{\mu}\mathcal{G}_{\mu}[\varepsilon\zeta,\beta b]\psi+\nabla(\varepsilon\zeta)\cdot\nabla\psi)^2}{2(1+\varepsilon^2\mu\vert\nabla\zeta\vert^2)}=0.
\end{array}
\right.
\label{Zakarov}
\end{equation}
It is a straightforward consequence of Green's identity that $$\frac{1}{\mu}\mathcal{G}_{\mu}[\varepsilon\zeta,\beta b]\psi=-\nabla\cdot (h v),$$
with $h=1+\varepsilon\zeta-\beta b$ and $v=\frac{1}{1+\varepsilon\zeta-\beta b}\int_{-1+\beta b}^{\varepsilon\zeta}\nabla\varphi(t,X,z)dz$.
 Therefore, the  first equation of (\ref{Zakarov}) exactly coincides 
with the first equation of (\ref{GN}). In order to derive the evolution equation on $v$,
 the key point is to obtain an asymptotic expansion of  $\nabla\psi$  with respect  to $\mu$ and in terms pf $v$ and $\zeta$.
As in \cite{LB}, we look for an asymptotic expansion of $\varphi$ under the form
\begin{equation}\label{exp of phi}
\varphi_{app}=\sum_{j=0}^{N}\mu^j\varphi_j.
\end{equation}
Plugging this expression into the boundary value problem (\ref{bvp}) one can cancel the residual up to the order $O(\mu^{N+1})$ provided that
\begin{equation}\label{phij}
\forall j=0,...,N,\qquad  \partial_z^2 \varphi_j = -\partial_x^2 \varphi_{j-1}-\partial_y^2 \varphi_{j-1} 
\end{equation}
(with the convention that $\varphi_{-1}=0$), together with the boundary conditions
\begin{equation}\label{phijbvp}
\forall j=0,...,N,\qquad 
\left\{
\begin{array}{ll}
\varphi_{j_{\mid_{z=\varepsilon\zeta}}}=\delta_{0,j}\psi,\\
\partial_z\varphi_j=\beta \nabla b\cdot\nabla\varphi_{j-1_{\mid_{z=-1+\beta b}}}
\end{array}
\right.
\end{equation}
(where $\delta_{0,j}=1$ if $j=0$ and $0$ otherwise).\\
By solving the ODE (\ref{phij}) with the boundary conditions (\ref{phijbvp}), one finds (see \cite{LB}) 
\begin{eqnarray}\label{phi0}
\varphi_0&=&\psi,\\
\varphi_1&=&(z-\varepsilon\zeta)\big(-\frac{1}{2}(z+\varepsilon\zeta)-1+\beta b\big)\Delta\psi+\beta(z-\varepsilon\zeta)\nabla b.\nabla \psi.\label{phi1}
\end{eqnarray}
According to formulae (\ref{phi0}), (\ref{phi1}), the horizontal component of the velocity in the fluid domain is given by 
$$V(z) = \nabla \varphi_0 (z) + \nabla \varphi_1 (z)  + O(\mu^2).$$
The  averaged velocity is thus given by 
 $$
 v=\dsp\frac{1}{h}\int_{-1+\beta b}^{\varepsilon\zeta}  (\nabla \varphi_0 (z) + \nabla \varphi_1 (z))\;dz + O(\mu^2),
 $$
 or equivalently 
$$
v=\nabla\psi-\mu\frac{1}{h}\mathcal{T}[h,\beta b]\nabla\psi+O(\mu^{2}),
$$
and thus 
\begin{equation}\label{exppsi}
\nabla\psi=v+\mu\frac{1}{h}\mathcal{T}[h,\beta b]\nabla\psi+O(\mu^{2}),
\end{equation}
where $\mathcal{T}[h,\beta b]$ is as in (\ref{defoper}).
As in \cite{LB}, taking the gradient of the second equation of (\ref{Zakarov}), replacing $\nabla\psi$ by its expression (\ref{exppsi})
and $\frac{1}{\mu}\mathcal{G}_{\mu}[\varepsilon\zeta,\beta b]\psi$ by $-\nabla\cdot (h v)$
in the resulting equation,  gives the standard  Green-Naghdi equations (after dropping the $O(\mu^2)$ terms),
 \begin{equation*}
\left\{
\begin{array}{lcl}
\dsp\partial_t\zeta+\nabla\cdot (hv)=0,\vspace{1mm}\\
\dsp\big(h+\mu\mathcal{T}[h,\beta b]\big)\partial_t v+h\nabla\zeta+\varepsilon \dsp\big(h+\mu\mathcal{T}[h,\beta b]\big)
( v\cdot \nabla )v\vspace{1mm}\\
\dsp\indent+\mu\varepsilon\big\lbrace\frac{2}{3}\nabla[h^3(\partial_1v\cdot\partial_2v^\perp+(\nabla\cdot v)^2)]+ \Re[h,\beta b](v)\big\rbrace=0,
\end{array}
\right.
\end{equation*}
where $h=1+\varepsilon\zeta- \beta b$, $v=(V_1,V_2)^T$, $v^\perp=(-V_2,V_1)^T$, and 
the linear operators  $\mathcal{T}[h,\varepsilon b]$ and $\Re[h,\beta b]$ being defined in (\ref{defoper}) and (\ref{defoper1}).

\subsection{Derivation of the new Green-Naghdi system  (\ref{GNnew})}\label{pres2}

It is quite common in the literature to work with variants of the Green-Naghdi equations
(\ref{GNnew}) that differ only up to terms of order $O(\mu^2)$ in order to improve the dispersive
properties of the model (see for instance \cite{WKGS,CLM}) or to change its mathematical
properties \cite{MGH}. Our approach here is in the same spirit since our goal is to derive
a new model with better energy estimates.\\
In order to  obtain the new Green-Naghdi system (\ref{GNnew}), let us remark  that,
from the expression  of $v$ one has
$$
\nabla\psi=v+\frac{\mu}{h}\mathcal{T}[h,\beta b]\nabla\psi+O(\mu^{2});
$$
and since $v=\nabla\psi+O(\mu),$ one gets the following Lemma:
\begin{lemma}\label{newformulate}
Let $v$ be the vertically averaged horizontal 
component of the velocity given above, one obtains
\begin{eqnarray}\label{definew v}
\hbox{\textnormal{curl}}\,\partial_tv&=&O(\mu),\nonumber\\
\hbox{\textnormal{curl}}\, (v\cdot\nabla)v&=&O(\mu).
\end{eqnarray}
\end{lemma}
\begin{remark}
For the sake of simplicity, we denote by $O(\mu)$ 
any  family of functions $(f_{\mu})_{0<\mu<1}$ 
such that $\big(\frac{1}{\mu}f_{\mu} \big)_{0<\mu<1}$
 remains bounded in $L^{\infty}([0,\frac{T}{\varepsilon}],H^{r}(\R^2))$, for some $r$ large enough.
\end{remark}
\begin{proof}
  By  applying the  operator $(\hbox{curl}\,\partial_t)(\cdot)$ to the identity  $v=\nabla\psi+O(\mu)$ one gets 
 $$
\hbox{curl}\,\partial_t v=O(\mu).
 $$
 In order to prove the second identity  of (\ref{definew v}), replace  $v=\nabla\psi+O(\mu)$ in $(v\cdot\nabla)v$ and
 apply the operator $\hbox{curl}\,(\cdot)$  to deduce 
 $$
\hbox{curl}\,(v\cdot\nabla)v=O(\mu).
 $$
\end{proof}
Using Lemma \ref{newformulate}, the quantities $\mu\nabla^\perp\hbox{curl}\,\partial_tv$ and 
$\mu\nabla^\perp\hbox{curl}\,\varepsilon(v\cdot\nabla)v$ 
are of  size $O(\mu^2)$, which is the precision of the GN equations (\ref{GN}). We can thus include these new terms in the second equation of (\ref{GN}) to get 
 \begin{equation*}
\left\{
\begin{array}{lcl}
\dsp\partial_t\zeta+\nabla\cdot (hv)=0,\vspace{1mm}\\
\dsp\Big(h+\mu\big(\mathcal{T}[h,\beta b]-\nabla^\perp\hbox{curl}\;\big)\Big)\partial_t v+h\nabla\zeta\vspace{1mm}\\\dsp\indent+\varepsilon \dsp\Big(h+\mu\big(\mathcal{T}[h,\beta b]-\nabla^\perp\hbox{curl}\;\big)\Big)
( v\cdot \nabla )v\vspace{1mm}\\
\dsp\indent+\mu\varepsilon\Big\lbrace\frac{2}{3}\nabla[h^3(\partial_1v\cdot\partial_2v^\perp+(\nabla\cdot v)^2)]+ \Re[h,\beta b](v)\Big\rbrace=O(\mu^2),
\end{array}
\right.
\end{equation*}
(the linear operators  $\mathcal{T}[h,\varepsilon b]$ and $\Re[h,\beta b]$ being defined in (\ref{defoper}) and (\ref{defoper1})).
\begin{remark}
We added the quantity  $\mu\nabla^\perp\hbox{\textnormal{curl}}\,\partial_tv=O(\mu^2)$  to the standard GN  equations (\ref{GN}) to obtain an energy  norm  $\vert \cdot\vert_{X^s}$:
$$
\vert (\zeta,v)\vert_{X^s}^2=\vert\zeta\vert_{H^s}^2+\vert v\vert_{H^s}^2
+\mu \vert \nabla\cdot v\vert_{H^s}^2+\mu \vert \hbox{\textnormal{curl}}\,v\vert_{H^s}^2.
$$ 
The last term is absent from the energy $\vert \cdot\vert_{Y^s}$ associated to the standard GN equations (\ref{GN}) (see \cite{AL2}). We will see in the next section that it
allows a control of the term  $\partial_1v\cdot\partial_2v^\perp$.
\end{remark}
\begin{remark}
 The  bilinear operator $\Re[h,\beta b](v)$ only involves second order derivatives of $v$  while third order derivatives of $v$ have been factorized by 
  $h+\mu \big(\mathcal{T}[h,\varepsilon b]-\nabla^\perp\hbox{curl}\;\big)$. The fact that the operator $\Re[h,\beta b](v)$   does not involve third
order derivatives is of great interest for the well-posedness of the new Green-Naghdi model.
\end{remark}
\section{Mathematical analysis of the new Green-Naghdi model}\label{wpnew}
\subsection{Preliminary results}\label{prnew}
For the sake of simplicity, we take here and throughout the rest of this paper  ($\beta=\varepsilon$) and we write 
$$\mfT=h+\mu \big(\mathcal{T}[h,\varepsilon b]-\nabla^\perp\hbox{curl}\;\big).$$ 
We always assume that  the nonzero depth condition 
\begin{equation}\label{depthcondnew}
\exists\; h_{min} >0, \quad \inf_{X\in \R^2} h\ge h_{min},\quad h=1+\varepsilon(\zeta- b)
\end{equation}
is valid  initially, which is a necessary condition for the new GN type system (\ref{GNnew}) to be  physically valid.
We shall demonstrate that the operator $\mfT$ plays an important role in the energy estimate and the local well-posedness of the GN type system 
(\ref{GNnew}). Therefore, we give here some of its properties.

The following lemma gives an important invertibility result on $\mfT$ and some properties of the inverse operator $\mfT^{-1}$.
\begin{lemma}\label{proprimnew'}
Let $b\in C_b^{\infty}(\R^2)$, $t_0> 1$  and $\zeta \in H^{t_0+1}(\R^2)$ 
be such that (\ref{depthcondnew}) is satisfied. Then, the operator 
$\mfT$ has a 
bounded inverse on $ (L^2(\R^2))^2$, and\\
\quad (i) For all $ 0\leq s\leq t_0+1$,  one has 
$$\quad\vert \mfT^{-1}f\vert_{H^s}+\sqrt{\mu}\vert \nabla\cdot\mfT^{-1}f\vert_{H^s}
+\sqrt{\mu}\vert \hbox{\textnormal{curl}}\,\mfT^{-1}f\vert_{H^s}\leq C(\frac{1}{h_{min}},\vert h-1 \vert_{H^{t_0+1}})\vert f\vert_{H^s},
$$
 and
 $$
\quad\sqrt{\mu}\vert \mfT^{-1}\nabla g\vert_{H^s}+\sqrt{\mu}\vert \mfT^{-1}\nabla^\perp g\vert_{H^s}\leq 
C(\frac{1}{h_{min}},\vert h-1 \vert_{H^{t_0+1}})
\vert g\vert_{H^s}.$$\\
\quad (ii) If  $s\geq t_0+1$ and $\zeta\in H^s(\R^2)$ then
$$
\parallel \mfT^{-1}\parallel_{(H^s)^2\rightarrow (H^s)^2}+
\sqrt{\mu}\parallel\mfT^{-1}\nabla \parallel_{(H^s)^2\rightarrow (H^s)^2}+\sqrt{\mu}\parallel\mfT^{-1}\nabla^\perp \parallel_{(H^s)^2\rightarrow (H^s)^2}\leq c_s,
$$
and
$$
\sqrt{\mu}\parallel\nabla\cdot\mfT^{-1} \parallel_{(H^s)^2\rightarrow (H^s)^2}+\sqrt{\mu}\parallel\hbox{\textnormal{curl}}\,\mfT^{-1} \parallel_{(H^s)^2\rightarrow (H^s)^2}\leq c_s,
$$
where $c_s$ is a constant depending on $\frac{1}{h_{min}}$,  $\vert h-1 \vert_{H^{s}}$
and independent of ($\mu$,$\varepsilon$) $\in (0,1)^2$.
\end{lemma}
\begin{remark}\label{remark b}
Here and throughout the rest of this paper, and for the sake of simplicity, we do not try to give some optimal regularity assumption on 
the bottom parametrization $b$. This could easily be done, but is of no interest for
our present purpose. Consequently, we omit to write the dependence on $b$ of the
different quantities that appear in the proof. 
\end{remark}
\begin{proof}
It can be remarked that the operator $\mfT$ is $L^2$ self-adjoint; since, one has
\begin{eqnarray*}
&&\Big(h+\mu \big(\mathcal{T}[h,\varepsilon b]-\nabla^\perp\hbox{curl}\;\big)v,v\Big)=(hv,v)\\
&&+\mu\Big(h\Big(\frac{h}{\sqrt{3}}\nabla\cdot v-\varepsilon\frac{\sqrt{3}}{2}\nabla b\cdot v\Big), \frac{h}{\sqrt{3}}\nabla\cdot v-\varepsilon\frac{\sqrt{3}}{2}\nabla b\cdot v\Big)+\frac{\mu\varepsilon^2}{4}(h\nabla b\cdot v,\nabla b\cdot v)\\
&&+\mu(\hbox{curl}\,v,\hbox{curl}\,v),
\end{eqnarray*}
and using the fact that $\inf_{\R^2} h\ge h_{min}$, one deduces that 
$$
(\mfT v,v)\ge E(\varepsilon b,v),
$$ 
with 
$$
 E(\varepsilon b,v):= h_{min}\vert v\vert_2^2+\mu h_{min}\vert \frac{h}{\sqrt{3}}\nabla\cdot v-\varepsilon\frac{\sqrt{3}}{2}\nabla b\cdot v \vert_2^2+\frac{\mu\varepsilon^2 h_{min}}{4}\vert \nabla b\cdot v\vert_2^2
 +\mu\vert \hbox{curl}\,v\vert_2^2,
$$
proceeding exactly as in the proof of the Lemma 1 of \cite{SI1}
it follows that $\mfT$ has an 
 inverse bounded on $ (L^2(\R^2))^2$.\\
For the rest of the proof, One can proceeding  as in the proof of Lemma 4.7  of \cite{AL2}, to get the result.
\end{proof}
\subsection{Linear analysis of (\ref{GNnew})}\label{lanew}
In order to rewrite the new GN type equations (\ref{GNnew}) in a condensed form, let us write $U=(\zeta,v^T)^T$, $v=(V_1,V_2)^T$  and
$$
Q=h^3(-\partial_1v\cdot\partial_2v^\perp-(\nabla\cdot v)^2).
$$
We decompose now the $O(\varepsilon\mu)$ term of the second equation of (\ref{GNnew}) under the form
\begin{equation*}
-\frac{2}{3}\nabla Q+ R[h,\varepsilon b](v)=R_1[U]v+r_2(U),
\end{equation*}
with for all  $f=(F_1,F_2)^T$
\begin{eqnarray}\label{decompose}
R_1[U]f&=& -\frac{2}{3}\nabla Q[U]f+ \frac{\varepsilon}{2}\nabla(h^2(V_1F_1\partial_1^2b+2V_1F_2\partial_1\partial_2 b+V_2F_2\partial_2^2b))\nonumber\\\nonumber\\&&
-\varepsilon h^2 (-\partial_1v\cdot\partial_2f^\perp-(\nabla\cdot v)(\nabla\cdot f))\nabla b;\nonumber\\\nonumber\\
r_2(U)&=&\varepsilon^2 h(V_1^2\partial_1^2b+2V_1V_2\partial_1\partial_2 b+V_2^2\partial_2^2b)\nabla b, \nonumber
\end{eqnarray}
where 
$$
Q[U]f=h^3(-\partial_1v\cdot\partial_2f^\perp-(\nabla\cdot v)(\nabla\cdot f)),
$$
(in particular, $Q=Q[U]v$).
The new Green-Naghdi equations (\ref{GNnew}) can thus be written after applying $\mfT^{-1}$ to both sides of the second equation
in  (\ref{GNnew}) as
\begin{equation}\label{condensedeqnew}
\partial_tU+A[U]U+B(U)=0,
\end{equation}
with $U=(\zeta,V_1,V_2)^T$, $v=(V_1,V_2)^T$ and where\\
\begin{equation}
A[U]=\left(
\begin{array}{cc}
\varepsilon v\cdot\nabla &h\nabla\cdot\\
\mfT^{-1}(h\nabla)& \varepsilon (v\cdot\nabla)+\varepsilon\mu\mfT^{-1}R_1[U]
\end{array}
\right)
\end{equation}
and\\
\begin{equation}
B(U)=\left(
\begin{array}{c}
\varepsilon \nabla b\cdot v\\
\varepsilon\mu \mfT^{-1}r_2(U)
\end{array}
\right).
\end{equation}
This subsection is devoted to the proof of energy estimates for the following initial value problem around some reference state 
$\underline{U}=(\underline{\zeta},\underline{V_1},\underline{V_2})^T$:
\begin{equation}\label{GNlsysnew}
	\left\lbrace
	\begin{array}{l}
	\dsp\partial_t U+A[\underline{U}]U+B (\underline{U})=0;
        \\
	\dsp U_{\vert_{t=0}}=U_0.
	\end{array}\right.
\end{equation}

We define first the $X^s$ spaces, which are the energy spaces for this problem.
\begin{definition}\label{defispace}
 For all $s\ge 0$ and $T>0$, we denote by $X^s$ the vector space $H^s(\R^2)\times (H^{s+1}(\R^2))^2$ endowed with the norm
$$
\forall\; U=(\zeta,v) \in X^s, \quad \vert U\vert^2_{X^s}:=\vert \zeta\vert^2 _{H^s}+\vert v\vert^2 _{(H^s)^2}+ \mu\vert \nabla \cdot v\vert^2 _{H^s}+\mu\vert \hbox{\textnormal{curl}}\,v\vert^2 _{H^s},
$$
while $X^s_T$ stands for $C([0,\frac{T}{\varepsilon}];X^{s})$ endowed with its canonical norm.
\end{definition}
We define the matrix $S$ as
\begin{equation}
S=\left( 
\begin{array}{cc}
 1& 0 \\ 
0& \underline{\mfT}
\end{array}
\right),
\end{equation}\\
with 
$\underline{h}=1+\varepsilon(\underline{\zeta}-b)$ and
$\underline{\mfT}=\underline{h}+\mu(\mathcal{T}[\underline{h},\varepsilon b]-\nabla^\perp(\nabla\wedge\;))$.
A natural energy for the IVP (\ref{GNlsysnew}) is given by
\begin{equation}\label{es}
 E^s(U)^2=(\Lambda^sU,S\Lambda^sU).
\end{equation}
The link between $ E^s(U)$ and the $X^s$-norm is investigated in the following Lemma.
\begin{lemma}\label{lemmaes}
Let $b\in C_b^{\infty}(\R^2)$,  $s\geq 0$ and $ \underline{\zeta}\in W^{1,\infty}(\R^2)$. Under the condition (\ref{depthcondnew}),
$E^s(U)$ is uniformly equivalent to the $\vert \cdot\vert_{X^s}$-norm  with respect to $(\mu, \varepsilon) \in (0,1)^2$:
$$
E^s(U) \leq C\big(\vert \underline{h}\vert_{W^{1,\infty}}\big)\vert U\vert_{X^s},
$$
and
$$
\vert U \vert_{X^s}\leq C\big(\frac{1}{h_{min}}\big) E^s(U).
$$
\end{lemma}
\begin{proof}
 Notice first that
$$
E^s(U)^2=\vert \Lambda^s\zeta\vert_{2}^2 +(\Lambda^sv,\underline{\mfT}\Lambda^sv),
$$
 one gets the first estimate using the explicit expression of $\underline{\mfT}$, integration by parts and Cauchy-Schwarz inequality. \\
The other inequality can be proved by using that $\inf_{x\in\R^2}h\ge h_{min}>0$ and proceeding as in the proof of  Lemma \ref{proprimnew'}. 
\end{proof}
We prove now the energy estimates in the following proposition. It is worth insisting on the
fact that these estimates are uniform with respect to $\varepsilon,\mu\in (0,1)$; since 
the control of the $s+1$ order derivatives by the $X^s$-norm disappears as $\mu\to0$, 
 the uniformity with respect to $\mu$ requires particular care (see the control of $B_{46}$ in the proof for
instance), but is very important for the application since $\mu\ll 1$.
\begin{proposition}\label{ESpropnew}
Let  $b\in C_b^{\infty}(\R^2)$,  $t_0>1$, $s\geq t_0+1$. Let also $\underline{U}=(\underline{\zeta}, \underline{V_1},\underline{V_2})^T$ $\in X^{s}_{T}$
 be such that $\partial_t \underline{U} \in X^{s-1}_{T}$ 
and satisfying the condition  (\ref{depthcondnew}) on $[0,\frac{T}{\varepsilon}]$. Then for all   $U_0\in X^{s}$
there exists a unique solution  $U=(\zeta, V_1, V_2)^T$ $\in X^{s}_{T} $ to (\ref{GNlsysnew}) and for all $0\leq t\leq\frac{T}{\varepsilon}$
$$
E^s(U(t))\leq e^{\varepsilon\lambda_{T} t}E^s(U_0)+\varepsilon \int^{t}_{0} e^{\varepsilon\lambda_T( t-t')}C(E^s(\underline{U})(t'))dt',
$$
for some $\lambda_{T}=\lambda_{T}(\sup_{0\leq t\leq T/\varepsilon}E^s(\underline{U}(t)),\sup_{0\leq t\leq T/\varepsilon}\vert\partial_t\underline{h}(t) \vert_{L^{\infty}})$ .
\end{proposition}
\begin{proof}
Existence and uniqueness of a solution to the IVP (\ref{GNlsysnew})  is achieved by using classical methods as in 
appendix A  of \cite{SI1} for the standard $1D$ Green-Naghdi equations and we thus focus our attention on the proof 
of the energy estimate. For any $\lambda \in \R$, we compute
$$
e^{\varepsilon\lambda t}\partial_t(e^{-\varepsilon\lambda t}E^s(U)^2)=-\varepsilon\lambda E^s(U)^2 +\partial_t(E^s(U)^2).
$$
Since
$$
 E^s(U)^2=(\Lambda^sU,S\Lambda^sU),
$$
and $U=(\zeta,V_1,V_2)^T$, $v=(V_1,V_2)^T$, we have
\begin{equation}
\partial_t(E^s(U)^2)=2(\Lambda^s\zeta,\Lambda^s\zeta_t)+2(\Lambda^sv,\underline{\mfT}\Lambda^sv_t)+(\Lambda^sv,[\partial_t,\underline{\mfT}]\Lambda^sv).
\end{equation}
One gets  using the equations  (\ref{GNlsysnew}) and integrating by parts,
\begin{eqnarray}\label{qtsctrnew}
\nonumber\lefteqn{\frac{1}{2}e^{\varepsilon\lambda t}\partial_t(e^{-\varepsilon\lambda t}E^s(U)^2)=-\frac{\varepsilon\lambda}{2}E^s(U)^2 
-(SA[\underline{U}]\Lambda^s U,\Lambda^s U) }\\
& &- \big(\big[\Lambda^s,A[\underline{U}]\big] U,S\Lambda^s U\big)â
-(\Lambda^s B(\underline{U}),S\Lambda^s U)+\frac{1}{2}(\Lambda^sv,[\partial_t,\underline{\mfT}]\Lambda^sv).
\end{eqnarray}
We now turn to bound from above the different components of the r.h.s of (\ref{qtsctrnew}).\\
$\bullet$ Estimate of $(SA[\underline{U}]\Lambda^s U,\Lambda^s U)$.
 Remarking  that
 \begin{equation*}
SA[\underline{U}]=\left(
\begin{array}{cc}
\varepsilon \underline{v}\cdot\nabla&\underline{h}\nabla\cdot\\
\underline{h}\nabla& \underline{\mfT}(\varepsilon \underline{v}\cdot\nabla)+\varepsilon\mu R_1[\underline{U}]
\end{array}
\right),
\end{equation*}
we get
\begin{eqnarray*}
(SA[\underline{U}]\Lambda^s U,\Lambda^s U)&=&  (\varepsilon\underline{v}\cdot\nabla\Lambda^s\zeta,\Lambda^s\zeta)+
(\underline{h}\nabla\cdot\Lambda^sv,\Lambda^s\zeta)\\&&
+(\underline{h}\nabla\Lambda^s \zeta,\Lambda^sv)+\big((\underline{\mfT}(\varepsilon\underline{v}\cdot\nabla)+\varepsilon\mu R_1[\underline{U}])\Lambda^s v,\Lambda^sv\big)\\&&
=: A_1+A_2+A_3+A_4.
\end{eqnarray*}
We now focus on the control of $(A_j)_{1\leq j\leq 4}$.\\
$-$ Control of $A_1$. Integrating by parts, one obtains 
$$
A_1= (\varepsilon\underline{v}\cdot\nabla\Lambda^s\zeta,\Lambda^s\zeta)=-\frac{1}{2}(\varepsilon\nabla\cdot\underline{v}\Lambda^s\zeta,\Lambda^s\zeta),
$$
and one can conclude by Cauchy-Schwarz inequality that
$$
\vert A_1\vert \leq\varepsilon C(\vert \nabla\cdot\underline{v}\vert_{L^{\infty}})E^s(U)^2.
$$
$-$ Control of $A_2+A_3$. First remark that 
$$
\vert A_2+A_3\vert=\vert(\nabla\underline{h}\cdot\Lambda^s v,\Lambda^s\zeta)\vert\leq \vert\nabla\underline{h} \vert_{(L^{\infty})^2}E^s(U)^2;
$$
we get,
$$
\vert A_2+A_3\vert\leq\varepsilon C(\vert\nabla\underline{h}\vert_{(L^{\infty})^2}E^s(U)^2.
$$
$-$ Control of $A_4$. One computes,
\begin{eqnarray*}
A_4&=&\varepsilon\big(\underline{\mfT} ((\underline{v}\cdot\nabla)\,\Lambda^s v),\Lambda^sv\big)+(\varepsilon\mu R_1[\underline{U}]\Lambda^s v,\Lambda^sv)\\&&
=:A_{41}+A_{42}.
\end{eqnarray*}
 Note first that
\begin{eqnarray*}
A_{41}&=&\varepsilon(\underline{h} \; (\underline{v}\cdot\nabla)\Lambda^s v,\Lambda^sv)+\frac{\varepsilon\mu}{3}(\underline{h}^3 \; \nabla \cdot(\underline{v}\cdot\nabla)\Lambda^s v),\Lambda^s\nabla\cdot v)\\&&
-\frac{\varepsilon^2\mu}{2}(\underline{h}^2 \; \nabla b\cdot(\underline{v} \cdot\nabla)\Lambda^s v,\Lambda^s\nabla\cdot v)-\frac{\varepsilon^2\mu}{2}(\underline{h}^2\nabla b \nabla \cdot(\underline{v}\cdot\nabla)\Lambda^s v,\Lambda^sv)\\&&
+\varepsilon^3\mu(\underline{h} \; \nabla b\nabla b\cdot(\underline{v}\cdot\nabla)\Lambda^s v,\Lambda^sv)+\varepsilon\mu(\hbox{curl}\,(\underline{v}\cdot\nabla)\Lambda^s v),\Lambda^s\hbox{curl}\,v);
\end{eqnarray*}
remark also that 
\begin{eqnarray*}
(\hbox{curl}\,(\underline{v}\cdot\nabla)\Lambda^s v),\Lambda^s\hbox{curl}\,v)&=&-\dsp\frac{1}{2}\Big((\hbox{curl}\,\Lambda^s v,\partial_1\underline{V_1}\hbox{curl}\,\Lambda^sv)+
(\nabla\wedge\Lambda^s v,\partial_2\underline{V_2}\nabla\wedge\Lambda^s v)\Big)\\&&
+(\hbox{curl}\,\Lambda^s v,\partial_1\underline{v}\cdot\nabla \Lambda^sV_2)-(\hbox{curl}\,\Lambda^s v,\partial_2\underline{v}\cdot\nabla\Lambda^sV_1),
\end{eqnarray*}
and that, for  all $F$  and $G$ smooth enough, one has
\begin{eqnarray*}
((G \cdot \nabla)v,F)=-(v,F  \;\nabla \cdot G)-(v,(G \cdot \nabla)F).
\end{eqnarray*}
By using successively the  above identities, the following relation (\ref{idofnorm})
\begin{equation}\label{idofnorm}
\vert \nabla F_1\vert_2^2+\vert \nabla F_2\vert_2^2=\vert \nabla\cdot F\vert_2^2+\vert \hbox{curl}\, F\vert_2^2,
\end{equation}
 integration by parts  and the Cauchy-Schwarz inequality, one obtains directly:
$$
\vert A_{41}\vert \leq \varepsilon C(\vert \underline{v}\vert_{(W^{1,\infty})^2},\vert \underline{\zeta}\vert_{W^{1,\infty}})E^s(U)^2.
$$
For $A_{42}$, remark that 
\begin{eqnarray*}
\vert A_{42}\vert&=&\vert(\varepsilon\mu R_1[\underline{U}]\Lambda^s v,\Lambda^sv) \vert\\
&=& \Big\vert +\frac{2}{3}\varepsilon\mu(Q[\underline{U}]\Lambda^s v,\Lambda^s\nabla\cdot v)-\mu\varepsilon^2\Big(\underline{h}^2  (-\partial_1\underline{v}\cdot\partial_2\Lambda^s v^\perp-(\nabla\cdot \underline{v})(\nabla\cdot \Lambda^s v))\nabla b,\Lambda^s v\Big)\\&&
-\frac{\mu\varepsilon^2}{2}\Big(\underline{h}^2(\underline{V}_1\Lambda^sV_1\partial_1^2b+2\underline{V}_1\Lambda^sV_2\partial_1\partial_2 b+\underline{V}_2\Lambda^sV_2\partial_2^2b),\Lambda^s\nabla\cdot v\Big)\Big\vert,
\end{eqnarray*}
where 
$$
Q[\underline{U}]f=\underline{h}^3(-\partial_1\underline{v}\cdot\partial_2f^\perp-(\nabla\cdot\underline{v})(\nabla\cdot f)).
$$
we deduce that
$$
\vert A_{42}\vert \leq \varepsilon C(\vert \underline{v}\vert_{(W^{1,\infty})^2},\vert \underline{\zeta}\vert_{W^{1,\infty}})E^s(U)^2.
$$
The estimates proved in $A_{41}$ and $A_{42}$ show that
$$
\vert A_{4}\vert \leq \varepsilon C(\vert \underline{v}\vert_{(W^{1,\infty})^2},\vert \underline{\zeta}\vert_{W^{1,\infty}})E^s(U)^2.
$$
$\bullet$ Estimate of $ \big(\big[\Lambda^s,A[\underline{U}]\big] U,S\Lambda^s U\big)$.
 Remark first that
 \begin{eqnarray*}
\lefteqn{\big(\big[\Lambda^s,A[\underline{U}]\big] U,S\Lambda^s U\big)=([\Lambda^s,\varepsilon \underline{v}]\cdot\nabla\zeta,\Lambda^s\zeta)+([\Lambda^s, \underline{h}]\nabla \cdot v,\Lambda^s\zeta)} \\
&+ &([\Lambda^s, \underline{\mfT}^{-1}\;\underline{h}]\nabla\zeta,\underline{\mfT}\Lambda^sv)+([\Lambda^s, \varepsilon (\underline{v}\cdot\nabla)] v,\underline{\mfT}\Lambda^sv)
+\varepsilon\mu\big(\big[\Lambda^s, \underline{\mfT}^{-1}R_1[\underline{U}]\big]v,\underline{\mfT}\Lambda^sv\big)\\
&=:& B_1+B_2+B_3+B_4+B_5.
\end{eqnarray*}
$-$ Control of $B_1+B_2=([\Lambda^s,\varepsilon \underline{v}]\cdot\nabla\zeta,\Lambda^s\zeta)+([\Lambda^s, \underline{h}]\nabla \cdot v,\Lambda^s\zeta)$.\\ 
Since $s\geq t_0+1$, we can use the following commutator estimate (\ref{com2}) (see e.g \cite{lannes'})
\begin{equation}\label{com2}
\vert[\Lambda^s,F]G\vert_2\lesssim \vert \nabla F\vert_{H^{s-1}}\vert G\vert_{H^{s-1}}.
\end{equation}
 to get 
$$
\vert B_1+B_2\vert  \leq \varepsilon C(E^s(\underline{U}))E^s(U)^2.
$$\\ 
$-$ Control of $B_4=([\Lambda^s, \varepsilon (\underline{v}\cdot\nabla)] v,\underline{\mfT}\Lambda^sv)$. By using  the explicit expression of  $\underline{\mfT}$ we get
\begin{eqnarray*}
B_4&=&([\Lambda^s, \varepsilon \underline{v}\cdot\nabla]v,\underline{h}\Lambda^s v)+\frac{\mu}{3}(\nabla\cdot[\Lambda^s, \varepsilon (\underline{v}\cdot\nabla)]v,\underline{h}^3 \; \Lambda^s \nabla\cdot v)\\&&
-\frac{\varepsilon\mu}{2}([\Lambda^s, \varepsilon (\underline{v}\cdot\nabla)]v,\underline{h}^2 \;  \nabla b\Lambda^s \nabla\cdot v)
+\frac{\varepsilon\mu}{2}([\Lambda^s, \varepsilon( \underline{v}\cdot\nabla)]v,\nabla(\underline{h}^2 \; \nabla b\cdot\Lambda^s v))
\\&&
+\varepsilon^2\mu([\Lambda^s, \varepsilon (\underline{v}\cdot\nabla)]v,\underline{h} \; \nabla b\nabla b\cdot\Lambda^s v)+
\mu(\hbox{curl}\,[\Lambda^s, \varepsilon (\underline{v}\cdot\nabla)]v, \Lambda^s\hbox{curl}\, v)\\
&:=&B_{41}+B_{42}+B_{43}+B_{44}+B_{45}+B_{46}.
\end{eqnarray*}
One obtains as for the control of $B_1$ and $B_2$ above that
$$
\vert B_{4j}\vert  \leq \varepsilon C(E^s(\underline{U}))E^s(U)^2, \quad j \in \{1,3,4,5\}.
$$\\ 
The control of $B_{42}$ and $ B_{46}$ is more delicate because of the dependence on $\mu$ 
(recall that the energy $E^s(U)$ controls $s+1$ derivatives of $v$, but with a small coefficient $\sqrt{\mu}$
in front of the derivatives of order $O(\sqrt{\mu})$). For $B_{46}$, we thus proceed as follows:
we first write
\begin{eqnarray*}
B_{46}&=&\mu(\hbox{curl}\,[\Lambda^s, \varepsilon (\underline{v}\cdot\nabla)]v, \Lambda^s \hbox{curl}\, v)\\
&=&\mu(\partial_1[\Lambda^s, \varepsilon \underline{V_1}\partial_1] V_2, \Lambda^s \hbox{curl}\, v)
+\mu(\partial_1[\Lambda^s, \varepsilon \underline{V_2}\partial_2]V_2, \Lambda^s \hbox{curl}\,v)\\
& &-\mu(\partial_2[\Lambda^s, \varepsilon \underline{V_1}\partial_1] V_1, \Lambda^s \hbox{curl}\,v)
-\mu(\partial_2[\Lambda^s, \varepsilon \underline{V_2}\partial_2]V_1, \Lambda^s \hbox{curl}\,v)\\
&:=& B_{461}+B_{462}+B_{463}+B_{464}.
\end{eqnarray*}
Remarking that
for all $j\in\{1,2\}$ we have
$$
\partial_j[\Lambda^s, f]g=[\Lambda^s, \partial_jf]g+[\Lambda^s, f]\partial_j g,
$$
and
$$
[\Lambda^s, f\partial_j]g=[\Lambda^s, f]\partial_j g,
$$
we can rewrite $B_{461}$ under the form
\begin{eqnarray*}
B_{461}&=&\mu(\partial_1[\Lambda^s, \varepsilon \underline{V_1}\partial_1] V_2, \Lambda^s \hbox{curl}\,v)\\
&=&\mu([\Lambda^s, \varepsilon \partial_1\underline{V_1}] \partial_1V_2, \Lambda^s\hbox{curl}\,v)+\mu([\Lambda^s, \varepsilon \underline{V_1}] \partial_1^2V_2, \Lambda^s \hbox{curl}\, v).
\end{eqnarray*}
It is then easy to use the  commutator estimate (\ref{com2}) in order to obtain 
$$
\vert B_{46j}\vert  \leq \varepsilon C(E^s(\underline{U}))E^s(U)^2, \quad j \in \{1,2,3,4\}.
$$\\ 
Similarly, $B_{42}$ is  controled by $ \varepsilon C(E^s(\underline{U}))E^s(U)^2$.
The estimates proved in \\ $(B_{4j})_{ j \in \{1,2,3,4,5,6\}}$ show that
$$
\vert B_{4}\vert  \leq \varepsilon C(E^s(\underline{U}))E^s(U)^2.
$$\\ 
$-$ Control of $B_3=([\Lambda^s, \underline{\mfT}^{-1}\;\underline{h}]\nabla\zeta,\underline{\mfT}\Lambda^sv)$. Remark first that
\begin{equation*}
\underline{\mfT}[\Lambda^s, \underline{\mfT}^{-1}]\underline{h}\nabla\zeta=\underline{\mfT}[\Lambda^s, \underline{\mfT}^{-1}\;\underline{h}]\nabla\zeta
-[\Lambda^s,\underline{h}]\nabla\zeta;
\end{equation*}
moreover, since $[\Lambda^s,\underline{\mfT}^{-1}]=-\underline{\mfT}^{-1}[\Lambda^s,\underline{\mfT}]\underline{\mfT}^{-1}$, one gets 
\begin{equation*}
\underline{\mfT}[\Lambda^s, \underline{\mfT}^{-1}\;\underline{h}]\nabla\zeta=-[\Lambda^s, \underline{\mfT}] \underline{\mfT}^{-1}\underline{h}\nabla\zeta
+[\Lambda^s,\underline{h}]\nabla\zeta,
\end{equation*}
and one can check by using  the explicit expression of  $\underline{\mfT}$ that
\begin{eqnarray*}
\underline{\mfT}[\Lambda^s, \underline{\mfT}^{-1}\;\underline{h}]\nabla\zeta&=& -[\Lambda^s, \underline{h}] \underline{\mfT}^{-1}\underline{h}\nabla\zeta
+\frac{\mu}{3}\nabla\{[\Lambda^s, \underline{h}^3]\nabla\cdot(\underline{\mfT}^{-1}\underline{h}\nabla\zeta)\} \\&&
-\frac{\varepsilon\mu}{2}
\nabla[\Lambda^s, \underline{h}^2\nabla b] \cdot\underline{\mfT}^{-1}\underline{h}\nabla\zeta
+\frac{\varepsilon\mu}{2}
[\Lambda^s, \underline{h}^2\nabla b] \nabla\cdot( \underline{\mfT}^{-1}\underline{h}\nabla\zeta) \\ &&
-\varepsilon^2\mu[\Lambda^s, \underline{h}\nabla b\nabla b^T] \underline{\mfT}^{-1}\underline{h}\nabla\zeta
+[\Lambda^s,\underline{h}]\nabla\zeta.
\end{eqnarray*}
 Integrating by parts yields
\begin{eqnarray*}
B_3&=&\big(\underline{\mfT}[\Lambda^s, \underline{\mfT}^{-1}\;\underline{h}]\nabla\zeta,\Lambda^s v\big)\\
&=& -\big([\Lambda^s, \underline{h}] \underline{\mfT}^{-1}\underline{h}\nabla\zeta,\Lambda^s v\big)-\frac{\mu}{3}\big(\{[\Lambda^s, \underline{h}^3]\nabla\cdot(\underline{\mfT}^{-1}\underline{h}\nabla\zeta)\},\Lambda^s\nabla\cdot v\big)\\
&& +\frac{\varepsilon\mu}{2}
\big([\Lambda^s, \underline{h}^2\nabla b] \cdot\underline{\mfT}^{-1}\underline{h}\nabla\zeta,\Lambda^s\nabla\cdot  v\big)
+\frac{\varepsilon\mu}{2}
\big([\Lambda^s, \underline{h}^2\nabla b] \nabla\cdot( \underline{\mfT}^{-1}\underline{h}\nabla\zeta),\Lambda^s v\big) \\&&
-\varepsilon^2\mu\big([\Lambda^s, \underline{h}\nabla b\nabla b^T] \underline{\mfT}^{-1}\underline{h}\nabla\zeta,\Lambda^s v\big)
+\big([\Lambda^s,\underline{h}]\nabla\zeta,\Lambda^s v\big).
\end{eqnarray*}
One deduces directly from Lemma \ref{proprimnew'}, the commutator estimate (\ref{com2}), and Cauchy-Schwarz inequality 
that
\begin{eqnarray*}
\vert B_3\vert & \leq& C\big(\frac{1}{h_{min}},\vert \underline{h}-1 \vert_{H^{s}}\big)\; \Big\{\Big( \vert\nabla\underline{h} \vert_{H^{s-1}}+\varepsilon^2\mu\vert\underline{h}^2\nabla b\nabla b^T \vert_{H^s}
+\frac{\varepsilon\mu}{2}\vert\underline{h}\nabla b  \vert_{H^s} \Big)\vert\underline{h}  \nabla\zeta\vert_{H^{s-1}}\\&&
+ \Big(\frac{1}{3}\vert\nabla\underline{h}^3\vert_{H^{s-1}}+ \frac{\varepsilon\sqrt{\mu}}{2}\vert\underline{h}^2\nabla b \vert_{H^s}\Big)
\vert\underline{h} \nabla\zeta\vert_{H^{s-1}} 
 +\vert\nabla\underline{h} \vert_{H^{s-1}} \vert \nabla\zeta\vert_{H^{s-1}}\Big)\Big\} \vert v\vert_{X^s}.
\end{eqnarray*}
Finally,  we deduce
$$
\vert B_3\vert  \leq \varepsilon C(E^s(\underline{U}))E^s(U)^2.
$$\\ 
$-$ Control of $B_5=\varepsilon\mu\big(\big[\Lambda^s, \underline{\mfT}^{-1}R_1[\underline{U}]\big]v,\underline{\mfT}\Lambda^sv\big)$. Let us first write 
\begin{equation*}
\underline{\mfT}\big[\Lambda^s, \underline{\mfT}^{-1}R_1[\underline{U}]\big]v=-[\Lambda^s, \underline{\mfT}] \underline{\mfT}^{-1}R_1[\underline{U}]v
+\big[\Lambda^s,R_1[\underline{U}]\big]v
\end{equation*}
so, that 
\begin{eqnarray*}
\underline{\mfT}\big[\Lambda^s, \underline{\mfT}^{-1}R_1[\underline{U}]\big]v&=& -[\Lambda^s, \underline{h}] \underline{\mfT}^{-1}R_1[\underline{U}]v
+\frac{\mu}{3}\nabla\{[\Lambda^s, \underline{h}^3]\nabla\cdot(\underline{\mfT}^{-1}R_1[\underline{U}]v)\} \\&&
- \frac{\varepsilon\mu}{2}
\nabla\{[\Lambda^s, \underline{h}^2\nabla b] \cdot\underline{\mfT}^{-1}R_1[\underline{U}]v\}+
\frac{\varepsilon\mu}{2}
[\Lambda^s, \underline{h}^2\nabla b] \nabla\cdot(\underline{\mfT}^{-1}R_1[\underline{U}]v)\\&&
-\varepsilon^2\mu[\Lambda^s, \underline{h}\nabla b\nabla b^T] \underline{\mfT}^{-1}R_1[\underline{U}]v
+[\Lambda^s,R_1[\underline{U}]]v.
\end{eqnarray*}
To control the term $\big(\big[\Lambda^s,R_1[\underline{U}]\big]v,\Lambda^sv\big)$ we use the explicit expression of $R_1[\underline{U}]$:\\
 \begin{eqnarray*}
 R_1[\underline{U}]f&=& -\frac{2}{3}\nabla Q[\underline{U}]f+ \frac{\varepsilon}{2}\nabla(\underline{h}^2(\underline{V_1}F_1\partial_1^2b+2\underline{V_1}F_2\partial_1\partial_2 b+
 \underline{V_2}F_2\partial_2^2b))\\&&
 -\varepsilon h^2 (-\partial_1\underline{v}\cdot\partial_2f^\perp -(\nabla\cdot\underline{ v})(\nabla\cdot f))\nabla b,
\end{eqnarray*}
where,
 $$
  Q[\underline{U}]f=\underline{h}^3(-\partial_1\underline{v}\partial_2f^\perp-(\nabla\cdot \underline{v})(\nabla\cdot f)).
 $$
As for the control of the  term $\big(\nabla\{[\Lambda^s, \underline{h}^3]\nabla\cdot(\underline{\mfT}^{-1}R_1[\underline{U}]v)\},\Lambda^sv\big)$ we use
the explicit expression of $R_1[\underline{U}]$, the relation (\ref{idofnorm}), the commutator estimate  (\ref{com2})  and 
Lemma \ref{proprimnew'}. Indeed,
\begin{eqnarray*}
\lefteqn{\big( \nabla\{[\Lambda^s, \underline{h}^3] \nabla \cdot(\underline{\mfT}^{-1}R_1[\underline{U}]v)\},\Lambda^sv\big)=
-\frac{2}{3}\big([\Lambda^s, \underline{h}^3] \nabla \cdot(\underline{\mfT}^{-1} \nabla Q[\underline{U}]v),\Lambda^s\nabla \cdot v\big)}\\
& -&\frac{\varepsilon}{2}\big([\Lambda^s, \underline{h}^3]\nabla \cdot\underline{\mfT}^{-1}\nabla(\underline{h}^2(\underline{V_1}V_1\partial_1^2b+2\underline{V_1}V_2\partial_1\partial_2 b
+ \underline{V_2}V_2\partial_2^2b)),\Lambda^s\nabla \cdot v\big)\\
&-&\varepsilon \big([\Lambda^s, \underline{h}^3]\nabla\cdot\underline{\mfT}^{-1}(h^2 (-\partial_1\underline{v}\cdot\partial_2\Lambda^s v^\perp -(\nabla\cdot\underline{ v})(\nabla\cdot v))\nabla b),\Lambda^s\nabla \cdot v\big).
\end{eqnarray*}
and thus, after remarking that 
\begin{eqnarray*}
\vert \nabla\cdot(\underline{\mfT}^{-1}\nabla Q[\underline{U}]v )\vert_{H^{s-1}}&\leq& \vert \underline{\mfT}^{-1}\nabla Q[\underline{U}]v\vert_{H^{s}},
\end{eqnarray*}
we can proceed as for the control of $B_3$ to get
$$
\vert B_5\vert  \leq \varepsilon C(E^s(\underline{U}))E^s(U)^2.
$$\\ 
$\bullet$ Estimate of $(\Lambda^s B(\underline{U}),S\Lambda^s U)$. Note first that 
\begin{equation*}
B(\underline{U})=\left( 
\begin{array}{c}
 \varepsilon \nabla b \cdot\underline{v} \\ 
\varepsilon\mu \underline{\mfT}^{-1} r_2(\underline{U})
\end{array}
\right)
\end{equation*}
so that
\begin{eqnarray*}
(\Lambda^sB(\underline{U}),S\Lambda^s U)&=&(\Lambda^s( \varepsilon \nabla b\cdot \underline{v}),\Lambda^s\zeta)
+(\Lambda^s( \underline{\mfT}^{-1} r_2(\underline{U})), \underline{\mfT}\Lambda^sv)\\
&=&(\Lambda^s( \varepsilon \nabla b \cdot\underline{v}),\Lambda^s\zeta)
-\varepsilon\mu\big([\Lambda^s, \underline{\mfT}] \underline{\mfT}^{-1} r_2(\underline{U}) ,\Lambda^sv\big)\\
&& +  \varepsilon\mu\big(\Lambda^s r_2(\underline{U}), \Lambda^sv\big). 
\end{eqnarray*}
Using again here the explicit expressions of $ \underline{\mfT}$, $r_2(\underline{U})$ and Lemma  \ref{proprimnew'}, we get
$$
(\Lambda^sB(\underline{U}),S\Lambda^s U) \leq \varepsilon C(E^s(\underline{U}))E^s(U).
$$
$\bullet$ Estimate of $(\Lambda^sv,[\partial_t,\underline{\mfT}]\Lambda^sv)$. We have that
\begin{eqnarray*}
(\Lambda^sv,[\partial_t,\underline{\mfT}]\Lambda^sv)&=&(\Lambda^sv,\partial_t\underline{h}\Lambda^sv)+\frac{\mu}{3}(\Lambda^s \nabla\cdot v,\partial_t\underline{h}^3\Lambda^s\nabla\cdot v)\\&&
-\frac{\varepsilon\mu}{2}(\Lambda^sv,\partial_t\underline{h}^2 \nabla b\Lambda^s\nabla\cdot v)-\frac{\varepsilon\mu}{2}(\Lambda^s\nabla\cdot v,\partial_t\underline{h}^2 \nabla b\cdot\Lambda^s v)\\&&
+\varepsilon^2\mu(\Lambda^sv,\partial_t\underline{h} \nabla b\nabla b^T \Lambda^sv).
\end{eqnarray*}
Controlling these terms by   $\varepsilon C(E^s(\underline{U}),\vert\partial_t\underline{h} \vert_{L^{\infty}})E^s(U)^2$
follows directly from a Cauchy-Schwarz inequality and an integration by 
parts.\\

Gathering the informations provided by the above estimates and using the fact that the embedding $H^s(\R^2)\subset W^{1,\infty}(\R^2)$ is continuous,  we get
$$
e^{\varepsilon\lambda t}\partial_t (e^{-\varepsilon\lambda t}E^s(U)^2) 
	\leq  \varepsilon\big(C(E^s(\underline{U}),\vert\partial_t\underline{h} \vert_{L^{\infty}})-\lambda\big)E^s(U)^2+\varepsilon C(E^s(\underline{U}))E^s(U).
$$
Taking $\lambda=\lambda_T$ large enough (how large depending on 
$\sup_{t\in [0,\frac{T}{\varepsilon}]}C(E^s(\underline{U}),\vert\partial_t\underline{h} \vert_{L^{\infty}})$
to have the first term of the right hand side negative for all $t\in [0,\frac{T}{\varepsilon}]$, one deduces
$$
	\forall t\in [0,\frac{T}{\varepsilon}],\qquad
	e^{\varepsilon\lambda t}\partial_t (e^{-\varepsilon\lambda t}E^s(U)^2) 
	\leq\varepsilon C(E^s(\underline{U}))E^s(U).
$$
Integrating this differential inequality yields therefore
$$
	\forall t\in [0,\frac{T}{\varepsilon}],\qquad
	E^s(U)\leq e^{\varepsilon\lambda_{T} t}E^s(U_0)+\varepsilon \int^{t}_{0} e^{\varepsilon\lambda_T( t-t')}C(E^s(\underline{U})(t'))dt',
$$
which  is the desired result.
\end{proof}
\subsection{Main result}\label{mrnew}
In this subsection we prove the main result of this paper, which shows
the well-posedness of the new Green-Naghdi equations (\ref{GNnew}) for times of order $O(\frac{1}{\varepsilon})$.
\begin{theorem}\label{th1}
	Let $b\in C_b^{\infty}(\R^2)$,  $t_0>1$, $s\geq t_0+1$. Let also 
	 $U_0=(\zeta_0,v_0^T)^T\in X^s$
	 be such that  (\ref{depthcondnew}) is satisfied. Then  there exists 
         a maximal $T_{max}>0$, uniformly bounded from below with respect to $\varepsilon,\mu\in (0,1)$, such that the new Green-Naghdi equations (\ref{GNnew}) admit 
	 a unique solution $U=(\zeta,v^T)^T\in X^s_{T_{max}}$ with the initial condition $(\zeta_0,v_0^T)^T$
         and preserving the nonvanishing depth condition (\ref{depthcondnew}) for any $t\in [0,\frac{T_{max}}{\varepsilon})$.
         In particular if $T_{max}<\infty$ one has
         $$ \vert U(t,\cdot)\vert_{X^s}\longrightarrow\infty\quad\hbox{as}\quad t\longrightarrow \frac{T_{max}}{\varepsilon},$$
         or
         $$ \inf_{\R^2} h(t,\cdot)=\inf_{\R^2}1+\varepsilon(\zeta(t,\cdot)-b(\cdot))\longrightarrow 0 \quad\hbox{as}\quad t\longrightarrow \frac{T_{max}}{\varepsilon}.$$
    \end{theorem}
\begin{remark}
For $2D$ surface waves, non flat bottoms, B. A. Samaniego and D. Lannes \cite{AL2} proved  a well-posedness result for  the standard $2D$ Green-Naghdi equations 
using a Nash-Moser scheme. Our result only uses a standard Picard iterative and there is therefore no loss of regularity of the solution of the new $2D$ Green-Naghdi 
equations  with respect to the initial condition.
\end{remark}
\begin{remark}
No smallness assumption on $\varepsilon$ nor $\mu$ is required in the theorem. The fact that $T_{max}$
is uniformly bounded from below with respect to these parameters allows us to say that if some
smallness assumption is made on $\varepsilon$, then the existence time becomes larger, namely of order $O(1/\varepsilon)$. 
\end{remark}
\begin{proof}
Using the energy estimate of Proposition \ref{ESpropnew}, one proves 
the result following the same lines as in the proof  of Theorem 1 in \cite{SI1}, which is 
itself an adaptation of the standard proof of 
well-posedness of hyperbolic systems (e.g \cite{AG, taylor}).
\end{proof}
\subsection{Conservation of the almost irrotationality of $v$}\label{conservirrot}
To obtain the new $2D$ Green-Naghdi  model (\ref{GNnew}) we used the fact that $ \hbox{curl}\,v=O( \mu)$, we prove in the following 
Theorem that the new model (\ref{GNnew}) conserves of course this property.
 \begin{theorem}
 Let   $t_0>1$, $s\geq t_0+1$,
$U_0=(\zeta_0,v_0^T)^T\in X^{s+4}$ with $\vert \hbox{\textnormal{curl}}\, v_0\vert_{H^s}\leq \mu C(\vert U_0\vert_{X^s})$. Then,
 the solution $U=(\zeta, v^T)^T\in X^{s+4}_{T_{max}}$  
of the new Green-Naghdi  equations (\ref{GNnew}) 
 with the initial condition $(\zeta_0,v_0^T)^T$ satisfies 
          $$
          \forall\,0<T<T_{max}, \quad \vert \hbox{\textnormal{curl}}\,v\vert_{L^{\infty}([0,\frac{T}{\varepsilon}],H^s)}\leq\mu\, C(T,\vert U_0\vert_{X^{s+4}}).
          $$
\end{theorem}
\begin{proof}
Applying the operator $ \hbox{curl}\,( \cdot)$ to the second equation of the model  (\ref{GNnew}) after multiplying it by $\dsp\frac{1}{h}$  yields 
$$
 \partial_t w+ \varepsilon \nabla\cdot(vw)= \varepsilon\mu f_1 +g,
$$
where
\begin{eqnarray*}
f_1&=&-\mbox{\textnormal{curl }}\big(\frac{1}{h}\mathcal{T}[h,\varepsilon b]-\frac{1}{h}\nabla^\perp\mbox{curl }\big)v\cdot \nabla v\\
& &-\mbox{curl }\frac{1}{h}\Big\lbrace\frac{2}{3}\nabla[h^3(\partial_1v\cdot\partial_2v^\perp+(\nabla\cdot v)^2)]+ \Re[h,\varepsilon b](v)\Big\rbrace
\end{eqnarray*}
and 
$$
g=-\mu\mbox{curl }\big(\frac{1}{h}\mathcal{T}[h,\varepsilon b]-\frac{1}{h}\nabla^\perp\mbox{curl }\big)\partial_t v.
$$
From the identity $\mbox{curl}(\frac{1}{h}W)=-\varepsilon \frac{1}{h^2}\nabla^\perp(\zeta-b)\cdot W+\frac{1}{h}\mbox{curl }W$, we deduce that $g$ can be put under the
form 
$$
g=\varepsilon\mu f_2+\mu\hbox{curl}\,(\frac{1}{h}\nabla^\perp\partial_t w),
$$
 with
$$
f_2=\frac{1}{h^2}\nabla^\perp(\zeta-b)\cdot \mathcal{T}[h,\varepsilon b]\partial_tv.
$$
We have thus shown that $w$ solves
$$
\Big(I-\mu\mbox{curl }\big(\frac{1}{h}\nabla^\perp\,\big)\Big)\partial_t w+\varepsilon \nabla\cdot(vw)=\varepsilon\mu f_1+\varepsilon\mu f_2.
$$
Energy estimates on this equation then show that for all $0<T<T_{max}$, one has
$$
\vert w\vert_{L^\infty([0,T/\varepsilon],H^{s})}\leq \mu C(T,\vert U\vert_{X^s_T})\big(\vert f_1\vert_{L^\infty([0,T/\varepsilon],H^{s})}+\vert f_2\vert_{L^\infty([0,T/\varepsilon],H^{s})}\big).
$$
Now, one deduces from the explicit expression of $f_j$ ($j=1,2$) that
$\vert f_j\vert_{L^\infty([0,T/\varepsilon],H^{s})}\leq C(\vert U\vert_{X^{s+4}_T})$; with Theorem \ref{th1}, we deduce that $\vert f_j\vert_{L^\infty([0,T/\varepsilon],H^{s})}\leq C(T,\vert U_0\vert_{X^{s+4}_T})$, and the result follows.
\end{proof}
\subsection*{Acknowledgments}\ The author is grateful to David Lannes for encouragement and many helpful discussions.

\providecommand{\href}[2]{#2}

\begin{thebibliography}{10}

\bibitem{AG} 
S.~Alinhac, P.~G\'erard, 
{\it Op\'erateurs pseudo-diff\'erentiels et tho\'er\`eme de Nash-Moser},
Savoirs Actuels. InterEditions, Paris; Editions du Centre National de la Recherche Scientifique (CNRS), Meudon, 1991. 190 pp. 


\bibitem{AL}
B.~Alvarez-Samaniego, D.~Lannes,   
Large time existence for $3D$ water-waves and asymptotics,
{\it Inventiones Mathematicae} {\bf 171} (2008), 485--541.

\bibitem{AL2}
B.~Alvarez-Samaniego, D.~Lannes,   
A Nash-Moser theorem for singular evolution equations. 
Application to the Serre and Green-Naghdi equations.  
{\it Indiana Univ. Math. J.} {\bf 57} (2008), 97-131.
 


\bibitem{CLM} F. Chazel, D. Lannes, F. Marche, {\it Numerical simulation of strongly nonlinear and dispersive waves using a Green-Naghdi model}, submitted.

\bibitem{craig}
W.~Craig, {\it An existence theory for water waves and the Boussinesq and the Korteweg-de Vries scaling limits.}
Commun. Partial Differ. Equations {\bf 10}, 787-1003 (1985).

\bibitem{craigsulem}
Craig, W. and Sulem, C. and Sulem, P.-L. 1992 {\it  Nonlinear modulation of gravity 
waves: a rigorous approach,} Nonlinearity {\bf 5}(2), 497-522. 




\bibitem{SI1}
S. Israwi,{	\it Large Time existence For 1D Green-Naghdi equations.} (2009)
[hal-00415875, version 1].



\bibitem{lannes}
D.~Lannes. {\it Well-posedness of the water waves equations,} J. Amer. Math. Soc.{\bf 18} (2005), 605-654.

\bibitem{lannes'}
 D.~ Lannes {\it Sharp Estimates for pseudo-differential operators with symbols of limited smoothness and commutators}, J. Funct. Anal. , {\bf 232} (2006), 495-539.
\bibitem{LB}  D.~ Lannes, P. Bonneton, {\it Derivation of asymptotic two-dimensional time-dependent equations for surface water wave propagation}, Physics of fluids {\bf 21} (2009).
\bibitem{li}
Y. A. Li, {\it A shallow-water approximation to the full water wave problem, Commun. Pure 
Appl. Math.} {\bf 59} (2006), 1225-1285. 


\bibitem{MGH} O. Le M\'etayer, S. Gavrilyuk,  S. Hank, {\it A numerical scheme for the Green-Naghdi model}, to appear in Journal of Computational Physics.

\bibitem{nalimov}
V.~I.~Nalimov, {\it The Cauchy-Poison problem.} (Russian) Dinamika Splo\v{s}n. Sredy Vyp. 
18 Dinamika Zidkost. so Svobod. Granicami,{\bf254}, (1974) 104-210.



\bibitem{WKGS} G.~ Wei, J.~T.~ Kirby, S.~T.~ Grilli, R.~ Subramanya, {\it A fully nonlinear
Boussinesq model for surface waves. Part 1. Highly nonlinear unsteady
waves}, J. Fluid Mech. {\bf 294} (1995), 71–92.

\bibitem{wu1}
S.~Wu, {\it Well-posedness in sobolev spaces of the full water wave problem in 2-D,} Invent. Math. {\bf 130} (1997), no. 1, 39-72.

\bibitem{wu2}
S.~Wu, {\it Well-posedness in sobolev spaces of the full water wave problem in 3-D,} J. Amer. Math. Soc. 
{\bf 12} (1999), no. 2, 445-495.

\bibitem{taylor}
Michael E.~Taylor, {\it Partial Differential Equations II}, Applied Mathematical Sciences Volume 116. Springer.

\bibitem{yoshihara}
H.~Yosihara, {\it Gravity waves on the free surface of an incompressible perfect fluid of finite depth.}
Publ. Res. Inst. Math. Sci. {\bf 18} (1982), no.1, 49-96.
 
 
\bibitem{zakarov}
Zakharov, V. E.  {\it Stability of periodic waves of finite amplitude on the 
surface of a deep fluid } (1968) J. Appl. Mech. Tech. Phys.,  {\bf 2}, 190-194. 




\end{thebibliography}
\end{document}